\documentclass[12pt]{amsart}

\usepackage[margin=1.25in]{geometry}
\usepackage{amsmath, amsfonts, amssymb, amscd, amsthm, mathtools, stmaryrd, mathrsfs, graphicx}
\usepackage{amsxtra}
\usepackage[abbrev,alphabetic]{amsrefs}
\usepackage{hyperref}
\usepackage{cleveref}
\usepackage{comment}
\usepackage[all,cmtip]{xy}
\usepackage{tikz-cd}
\usetikzlibrary{cd}
\usepackage{booktabs}
\usepackage{multirow}
\usepackage{xcolor, soul}
\setstcolor{red}
\tikzcdset{
  cells={font=\everymath\expandafter{\the\everymath\displaystyle}},
}

\newtagform{tiny}{\tiny(}{)}

\makeatletter
\def\@tocline#1#2#3#4#5#6#7{\relax
  \ifnum #1>\c@tocdepth \else
    \par \addpenalty\@secpenalty\addvspace{#2}%
    \begingroup \hyphenpenalty\@M
    \@tempdima\if@empty{#4}{\csname r@tocindent\number#1\endcsname}{#4}\relax
    \parindent\z@ \leftskip#3\relax \advance\leftskip\@tempdima\relax
    \rightskip\@pnumwidth plus4em \parfillskip-\@pnumwidth
    #5\leavevmode\hskip-\@tempdima
    \ifcase #1 \or\or \hskip 1em \or \hskip 2em \else \hskip 3em \fi%
    #6\nobreak\hfill\hbox to\@pnumwidth{\@tocpagenum{#7}}\par
    \nobreak
    \endgroup
  \fi}
\makeatother



\title{On the boundedness of $k$-algebra homomorphisms between $k$-affinoid algebras}
\author{Shou Yoshikawa}
\address{Institute of Science Tokyo, Tokyo 152-8551, Japan}
\email{yoshikawa.s.9fe9@m.isct.ac.jp}
\newcommand{\Z}{\mathbb{Z}} \newcommand{\Q}{\mathbb{Q}} \newcommand{\R}{\mathbb{R}}

\newcommand{\cM}{\mathcal{M}}

\DeclareMathOperator{\Spec}{Spec}

\DeclareMathOperator{\Hom}{Hom}

\newcommand{\wt}{\widetilde}

\newcommand{\cond}[1]{\textup{(#1)}}

\newcommand{\eqtag}[1]{\overset{(\star_{#1})}{=}}

\theoremstyle{plain}
\newtheorem{theorem}{Theorem}[section]

\newtheorem{proposition}[theorem]{Proposition}
\newtheorem{lemma}[theorem]{Lemma}

\newtheorem{claim}[theorem]{Claim}

\newtheorem*{claim*}{Claim}
\newtheorem{theoremA}{Theorem}

\theoremstyle{definition}
\newtheorem{definition}[theorem]{Definition}

\newtheorem{example}[theorem]{Example}
\newtheorem{notation}[theorem]{Notation}
\newtheorem*{setup*}{Setup}

\theoremstyle{remark}

\newtheorem*{ackn}{Acknowledgements}

\numberwithin{equation}{section}

\crefname{theorem}{Theorem}{Theorems}
\crefname{proposition}{Proposition}{Propositions}
\crefname{lemma}{Lemma}{Lemmas}
\crefname{corollary}{Corollary}{Corollaries}
\crefname{conjecture}{Conjecture}{Conjectures}
\crefname{claim}{Claim}{Claims}
\crefname{notation}{Notation}{Notations}
\crefname{theoremA}{Theorem}{Theorems}
\crefname{corollaryA}{Corollary}{Corollaries}
\crefname{example}{Example}{Examples}

\begin{document}

\begin{abstract}
Let $k$ be a complete non-archimedean non-trivial valued field. In this paper, we investigate whether every $k$-algebra homomorphism between $k$-affinoid algebras is automatically bounded. We show that this property holds if and only if either  $|k^\times|^\Q = \R_{>0}$ holds, or $k$ has positive characteristic and is $F$-finite. 
\end{abstract}

\maketitle

\setcounter{tocdepth}{1}

\section{Introduction}

% The theory of $k$-affinoid algebras plays a central role in non-archimedean analytic geometry, especially in the framework introduced by Tate and further developed by Berkovich and others. A fundamental feature of these algebras is the existence of a natural norm that gives them a Banach algebra structure. In this context, one often considers homomorphisms of $k$-algebras that are compatible with these norms. However, not every $k$-algebra homomorphism between affinoid algebras is necessarily bounded with respect to the Banach norms.

In this paper, we prove the following theorem concerning the boundedness of $k$-algebra homomorphisms between $k$-affinoid algebras.

\begin{theoremA}\label{main-thm}\textup{(cf.~\cref{reduced-case})}
Let $k$ be a complete non-archimedean valued field of characteristic $p \geq 0$ with $|k^\times| \neq \{1\}$.
Then every $k$-algebra homomorphism between reduced $k$-affinoid algebras is bounded.
Moreover, the following conditions are equivalent:
\begin{enumerate}
    \item Every $k$-algebra homomorphism between $k$-affinoid algebras is bounded.
    \item Either $|k^\times|^{\mathbb{Q}} = \mathbb{R}_{>0}$, or $p > 0$ and $[k : k^p] < \infty$,
\end{enumerate}
where $k^p:=\{a^p \mid a\in k\}$.
\end{theoremA}

\noindent
In \cite{Temkin10}*{Fact~3.1.4.2}, condition \textup{(1)} in \cref{main-thm} is stated as always holding, but \cref{main-thm} shows that this is not the case in general.
The following example provides a concrete counterexample.

\begin{example}
Let $k := \mathbb{Q}_p$ for a prime number $p$, and take $r \in \mathbb{R}_{>0} \setminus |k^\times|^{\mathbb{Q}}$.
By \cref{computation-derivation}, there exists a non-zero derivation
\[
\delta \in \mathrm{Der}_{k[T]}(k\{r^{-1}T\}, K_r).
\]
This derivation induces a $k$-algebra homomorphism
\[
\varphi \colon k\{r^{-1}T\} \hookrightarrow K_r \to K_r\{T\}/(T^2), \quad g \mapsto g + \delta(g) T,
\]
which is not bounded.

Indeed, let $f \in k\{r^{-1}T\}$ with $\delta(f) \neq 0$, and take a sequence $\{f_n \in k[T]\}$ converging to $f$ in $k\{r^{-1}T\}$.
Then
\[
\lim \varphi(f_n) = \lim (f_n + \delta(f_n) T) = f \neq f + \delta(f) T = \varphi(f),
\]
so $\varphi$ is not continuous, hence not bounded.
\end{example}

The paper is organized as follows.
In Subsection~2.2, we prove the first assertion of \cref{main-thm} (see \cref{reduced-case}). To relate the norm with algebraic properties, we introduce the notion of a compatible system of $p$-power roots. This proof is inspired by the hint to Exercise~(ii) in \cite{Temkin10}.
In Subsection~2.3, we apply the result of Subsection~2.2 to reinterpret \cref{main-thm} as a deformation problem. This viewpoint allows us to connect the boundedness of $k$-algebra homomorphisms to the existence of non-zero derivations over polynomial rings, thereby linking it to algebraic properties of the base field~$k$.

\begin{ackn}
This result was established as part of a seminar on Berkovich spaces. The author would like to express his sincere gratitude to the main members of the seminar—Yusuke Matsuzawa, Masaru Nagaoka, and Teppei Takamatsu—for their valuable discussions and contributions during the seminar. 
The author was supported by JSPS KAKENHI Grant number JP24K16889.
\end{ackn}

\section{Proof of \cref{main-thm}}

\subsection{Notation}\label{ss-notation}
We follow the notation and terminology used in \cite{Temkin10} and \cite{Temkin04}.
This subsection provides additional remarks and clarifications regarding the symbols and terms employed throughout the paper.

\begin{enumerate}
\item Let $(A, \|\cdot\|_A)$ and $(B, \|\cdot\|_B)$ be seminormed rings.
A ring homomorphism $\varphi \colon A \to B$ is said to be \emph{bounded} if there exists a constant $C > 0$ such that $\|\varphi(a)\|_B \leq C \|a\|_A$ for all $a \in A$ \cite[Definition~2.1.1(i)]{Temkin10}.
If $\|\cdot\|_B$ is power-multiplicative, then the boundedness of $\varphi$ implies $\|\varphi(a)\|_B \leq \|a\|_A$ for all $a \in A$.

\item Let $(k, |\cdot|)$ be a complete non-archimedean valued field.
We use the terms \emph{strictly $k$-affinoid algebra} and \emph{$k$-affinoid algebra} in the sense of \cite[Definition~3.1.1.1(i)]{Temkin10}.
If $A$ is a $k$-affinoid algebra, then the structure morphism $(k, |\cdot|) \to (A, \rho_A)$ is bounded.
In particular, the spectral radius $\rho_A$ is non-archimedean.

\item Let $(k, |\cdot|)$ be a complete non-archimedean valued field, and let $A$ be a $k$-affinoid algebra.
We use the notions of the \emph{reduction} $\widetilde{A}$ of $A$ and the \emph{reduction map} $\pi_A \colon \cM(A) \to \Spec_{\R_{>0}} \widetilde{A}$ as defined in \cite[Section~3]{Temkin04}.
\end{enumerate}

\subsection{Preliminary}
In this subsection, we collect some basic properties about affinoid algebras will be used in the paper. 

\begin{proposition}\label{int-min-prime}
Let \( G \) be a group and \( A \) a \( G \)-graded ring. 
If \( A \) is graded reduced, that is, every homogeneous nilpotent element is zero, then the intersection of all graded minimal prime ideals of \( A \) is zero.
\end{proposition}

\begin{proof}
Let \( I \) be the intersection of all graded minimal prime ideals of \( A \). 
Then \( I \) is a graded ideal. 
Let \( f \in I \) be a homogeneous element. 
Then the localization \( G \)-graded ring \( A_f \), in the sense of \cite{Temkin04}*{Section~1}, satisfies \( \Spec_G(A_f) = \emptyset \), and in particular, \( A_f = \{0\} \).
This implies that there exists an integer \( n \geq 1 \) such that \( f^n = 0 \). 
Since \( A \) is graded reduced, it follows that \( f = 0 \), as desired.
\end{proof}

\begin{proposition}\label{Shilov-boundary}\textup{(\cite{Temkin04}*{Proposition~3.3})}
Let $k$ be a complete non-archimedean valued field with $|k^\times| \neq \{1\}$, and let $A$ be a $k$-affinoid algebra.
Let $\Gamma(A)$ denote the Shilov boundary of $A$ \textup{(see \cite{Temkin10}*{Fact~3.4.1.6~(ii)})}.
Then the following statements hold:
\begin{enumerate}
    \item Let $f \in A$. If $|f(x)| = 0$ for all $x \in \Gamma(A)$, then $\rho_A(f) = 0$.
    \item For every point $x$ in $\Gamma(A)$, there exists a bounded $k$-algebra homomorphism
    \[
    \psi_x \colon A \to A_x
    \]
    to a $k$-affinoid algebra $A_x$ such that:
    \begin{itemize}
        \item $A_x$ is reduced;
        \item the spectral seminorm $\rho_{A_x}$ is multiplicative;
        \item and $\psi_x^*(\rho_{A_x}) = x$ as seminorms on $A$.
    \end{itemize}
\end{enumerate}
\end{proposition}

\begin{proof}
We begin with the proof of (1).  
By definition, $\widetilde{A}$ is graded reduced.
Since $|f(x)| = 0$ for all $x \in \Gamma(A)$, it follows that $\widetilde{f} \in \widetilde{A}$ is contained in all graded minimal primes of $\widetilde{A}$.
Thus, $\widetilde{f} = 0$ by \cref{int-min-prime}, and consequently, $\rho_A(f) = 0$.

Next, we prove (2).  
Fix $x \in \Gamma(A)$ and let $\mathfrak{p} := \pi_A(x) \in \Spec_G(\widetilde{A})$.
Choose a homogeneous element $\widetilde{f} \in \mathfrak{p}$ such that $\widetilde{f}$ does not lie in any graded minimal prime of $\widetilde{A}$ other than $\mathfrak{p}$, and let $f \in A$ be a lift of $\widetilde{f}$.
Set $\psi_x \colon A \to A_x := A\{r^{-1}f\}^{\mathrm{red}}$, where $r := \rho_A(f)$.
Then, by \cite{Temkin04}*{Proposition~3.1~(ii)}, we have
\[
\widetilde{A_x} \simeq \widetilde{A\{r^{-1}f\}} \simeq \widetilde{A}_{\widetilde{f}}.
\]
Therefore, $\widetilde{A_x}$ has a unique graded minimal prime corresponding to $\mathfrak{p}$.
In particular, $\widetilde{A_x}$ is a graded domain, and hence $\rho_{A_x}$ is multiplicative.
Furthermore, we have
\[
\pi_A \circ \psi_x^*(\rho_{A_x}) = \widetilde{\psi_x}^{*} \circ \pi_{A_x}(\rho_{A_x}) = \mathfrak{p},
\]
and thus $\psi_x^*(\rho_{A_x}) = x$ by \cite{Temkin04}*{Proposition~3.3~(ii)}, as required.
\end{proof}

\begin{proposition}\label{A-to-AK}
Let $(A, \|\cdot\|)$ be a non-archimedean Banach ring, and let $r := (r_1, \ldots, r_n) \in \R^n_{>0}$.  
Let $\rho$ and $\rho_r$ denote the spectral radii on $A$ and on $A_r := A\{r^{-1}T, rT^{-1}\}$, respectively.  
Then for any $f = \sum a_\nu T^\nu \in A_r$, we have
\[
\rho_r(f) = \max\{\rho(a_\nu) r^\nu \mid \nu \in \Z^n\}.
\]
\end{proposition}

\begin{proof}
Using the bounded isomorphism
\[
A\{r^{-1}T, rT^{-1}\} \simeq A\{r_1^{-1}T_1, \ldots, r_n^{-1}T_n, r_1T_1^{-1}, \ldots, r_nT_n^{-1}\},
\]
we may reduce to the case $n = 1$. Write $r := r_1$ and $T := T_1$.
For $g = \sum b_i T^i \in A\{r^{-1}T, rT^{-1}\}$, define
\[
\sigma(g) := \max_i \rho(b_i) r^i.
\]

\noindent
\textbf{Step 1:} Let $f = aT^i$ for some $a \in A$ and $i \in \Z$. Then
\[
\rho_r(f) = \lim_{l \to \infty} \|f^l\|_{A_r}^{1/l} = \lim_{l \to \infty} \|a^l\|^{1/l} r^i 
= \rho(a) r^i = \sigma(f).
\]

\noindent
\textbf{Step 2:} Let $f \in A[T, T^{-1}]$, and define
\[
I' := \{i \in \Z \mid \rho(a_i) r^i = \sigma(f)\}, \quad 
f' := \sum_{i \in I'} a_i T^i, \quad 
f'' := \sum_{i \notin I'} a_i T^i.
\]
Then $\sigma(f) = \sigma(f')$ and $\sigma(f'') < \sigma(f)$. We claim:

\begin{claim}\label{cl:A-to-AK}
We have $\rho_r(f') = \sigma(f')$.
\end{claim}

\begin{proof}[Proof of \cref{cl:A-to-AK}]
Since $\rho_r$ is non-archimedean as $\rho_r \leq \|\cdot\|_{A_r}$, we have
\[
\rho_r(f') \leq \max_{i \in I'} \rho_r(a_i T^i) = \sigma(f') \quad \text{(by Step 1)}.
\]
To show the reverse inequality, let $i_0 \in I'$ be the largest integer such that $a_{i_0} \neq 0$. Then
\[
\rho_r(f') 
= \lim_{l \to \infty} \|(f')^l\|_{A_r}^{1/l} 
\geq \lim_{l \to \infty} \|a_{i_0}^l\|^{1/l} r^{i_0} = \rho(a_{i_0}) r^{i_0} = \sigma(f').
\]
Thus, $\rho_r(f') = \sigma(f')$.
\end{proof}

\noindent
By the same argument in the proof of \cref{cl:A-to-AK}, we have $\rho_r(f'') \leq \sigma(f'') < \sigma(f')=\rho_r(f')$, so
\[
\rho_r(f) = \rho_r(f') = \sigma(f') = \sigma(f).
\]

\noindent
\textbf{Step 3:} For a general $f \in A_r$, take a sequence $\{f_n\} \subset A[T, T^{-1}]$ such that $\lim_{n \to \infty} f_n = f$ in $A_r$.
Since both $\rho_r$ and $\sigma$ are continuous, we have
\[
\rho_r(f) = \lim \rho_r(f_n) = \lim \sigma(f_n) = \sigma(f),
\]
as desired.
\end{proof}

\subsection{The reduced case}
In this subsection, we prove the first assertion of \cref{main-thm}. To relate the norm with algebraic properties, we introduce the notion of a compatible system of $p$-power roots. This proof is inspired by the hint to Exercise~(ii) in \cite{Temkin10}.

\begin{definition}\label{defn:p-root}
Let $A$ be a ring , $f$ an element of $A$, and $p$ a prime number.
A sequence $\{f^{1/p^e} \in A\}_{e \geq 0}$ is called \emph{a compatible system of $p$-power roots of $f$} if $(f^{1/p^{e+1}})^p=f^{1/p^e}$ for all $e \geq 0$ and $f^{1/p^0}=f$.
\end{definition}

\begin{lemma}\label{p-power-root-to-unit}
Let $A$ be a Noetherian domain and let $f \in A$ be a non-zero element.
If there exists a compatible system of $p$-power roots of $f$, then $f$ is a unit in $A$.
\end{lemma}

\begin{proof}
Consider the ascending chain of principal ideals:
\[
(f) \subseteq (f^{1/p}) \subseteq (f^{1/p^2}) \subseteq \cdots.
\]
Since $A$ is Noetherian, this chain stabilizes. That is, there exists an integer $e \geq 0$ such that
\[
(f^{1/p^e}) = (f^{1/p^{e+1}}).
\]
In particular, there exists an element $g \in A$ such that
\[
f^{1/p^{e+1}} = g \cdot f^{1/p^e}.
\]
Raising both sides to the $p^{e+1}$-th power yields
\[
f = g^{p^{e+1}} \cdot f^p.
\]
Since $A$ is a domain and $f \neq 0$, we can cancel $f$ to obtain
\[
f^{p-1} \cdot g^{p^{e+1}} = 1.
\]
Thus, $f$ is invertible in $A$, as claimed.
\end{proof}

\begin{lemma}\label{exists-p-th-root}
Let $(A,|\cdot|)$ be a non-archimedean Banach ring such that $|\cdot|$ is multiplicative, $f$ an element of $A$, and $g$ a unit of $A$.
Let $p$ be a prime number with $|p|=1$.
We assume that there exists a $p$-th root $g^{1/p} \in A$ of $g$ and $|f-g| < |f|$.
Then there exists a $p$-th root $f^{1/p} \in A$ of $f$ with $|f^{1/p}-g^{1/p}|<|f^{1/p}|$.
\end{lemma}

\begin{proof}
Since we have $|f-g|<|f|$ and $|\cdot|$ is non-archimedean, it follows that $|f|=|g|$.
Since $|\cdot|$ is multiplicative, we obtain
\[
|g^{-1}f-1|=|g^{-1}|\cdot |f-g| < |g^{-1}f|=1.
\]
Therefore, if there exists a $p$-th root $(g^{-1}f)^{1/p}$ of $g^{-1}f$ with $|(g^{-1}f)^{1/p}-1|<1$, then $f^{1/p}:=g^{1/p}(g^{-1}f)^{1/p}$ is a $p$-th root of $f$ with
\[
|f^{1/p}-g^{1/p}|=|g^{1/p}|\cdot |(g^{-1}f)^{1/p}-1|<|g^{1/p}|.
\]
In conclusion, we may assume that $g=1$.
Next, we prove the following assertion:
\begin{claim}\label{cl:exist-p-root}
There exist $\{g_m \in A\}_{m \geq 1},\{h_m \in A\}_{m\geq 1}$ such that
\begin{enumerate}
    \item $(1+g_1+\cdots+g_m)^p=f+h_m$,
    \item $|g_1|=|f-1|$, 
    \item $|g_m| \leq |g_1|^m$, and
    \item $|h_m| \leq |g_1|^{m+1}$
\end{enumerate}
for every $m \geq 1$.
\end{claim}
\begin{proof}[Proof of \cref{cl:exist-p-root}]
We prove \cref{cl:exist-p-root} by induction on $m$.
For $m=1$, we set $g_1:=\frac{1}{p}(f-1)$ and $h_1:=(1+g_1)^p-f$.
Then conditions (1), (2), (3) are satisfied, thus we check the condition (4).
We have
\begin{align*}
    |h_1|
    &=|(1+g_1)^p-f|=|1+pg_1-f +\sum_{i= 2}^{p} \binom{p}{i} g_1^{i}| \\
    &=|\sum_{i= 2}^{p} \binom{p}{i} g_1^{i}| = |g_1|^2|\sum_{i= 2}^{p} \binom{p}{i} g_1^{i-2}| \leq |g_1^2|,
\end{align*}
thus condition (4) is satisfied.
Next, we assume that there exist $h_1,\ldots,h_{m}$, $g_1,\ldots,g_{m}$ satisfying (1)--(4).
We set $g_{m+1}:=-\frac{h_m}{p}$ and $h_{m+1}:=(1+g_1+\cdots+g_{m+1})^p-f$.
Then we have
\[
|g_{m+1}|=|h_m| \leq |g_1|^{m+1}
\]
by condition (4), condition (3) satisfied.
Finally, we check condition (4).
We have
\begin{align*}
    h_{m+1}
    &= (1+g_1+\cdots+g_{m+1})^p-f \\
    &=(1+g_1+\cdots+g_m)^p+p(1+g_1+\cdots+g_m)^{p-1}g_{m+1}+g_{m+1}^2\alpha -f \\
    &\eqtag{1} h_m+p(1+g_1+\cdots+g_m)^{p-1}g_{m+1}+\alpha \\
    &= h_m+pg_{m+1}+\beta+\alpha \\
    &=\beta+g_{m+1}^2\alpha, 
\end{align*}
where
\begin{align*}
    \alpha&=\sum_{i=2}^p \binom{p}{i}(1+g_1+\cdots+g_m)^{p-i}g_{m+1}^{i}, \\
    \beta&= pg_{m+1}\left((1+g_1+\cdots+g_{m})^{p-1}-1\right),
\end{align*}
and $(\star_1)$ follows from the condition (1).
Since $|\alpha| \leq |g_{m+1}|^2 \leq |g_1|^{m+2}$ and $|\beta| \leq |g_1g_{m+1}| \leq |g_1|^{m+2}$, we obtain $|h_{m+1}| \leq |g_1|^{m+2}$, as desired.
\end{proof}
We take $\{g_m\}_{m \geq 1}$ and $\{h_m\}_{m \geq 1}$ as in \cref{cl:exist-p-root}.
We set $f^{1/p}:=1+\sum_{i=1}^\infty g_i$, where the existence follows from $|g_1|=|f-1|<1$ and condition (3) in \cref{cl:exist-p-root}.
Thus, we have 
\[
(f^{1/p})^p \eqtag{2} \lim_{m \to \infty}(f+h_m) \eqtag{3} f,
\]
where $(\star_2)$ follows from condition \cond{1} and $(\star_3)$ follows from condition \cond{4}.
In particular, $f^{1/p}$ is a $p$-th root of $f$.
Furthermore, we have
\[
|f^{1/p}-1| \leq |g_1| <1,
\]
thus $f^{1/p}$ satisfies the desired conditions.
\end{proof}

\begin{notation}\label{notation-non-triv}
Let $(k,|\cdot|)$ be a complete non-archimedean valued field with $|k^\times| \neq \{1\}$.
Let $p$ be a prime number with $|p|=1$.
\end{notation}

\begin{lemma}\label{norm-f-in}
We use \cref{notation-non-triv}.
We assume $k$ is an algebraically closed field.
Let $(A,||\cdot||_A)$ be a $k$-affinoid algebra and $f \in A \backslash \{0\}$.
If there exits a compatible system of $p$-power roots $\{f^{1/p^e}\}$, then $\rho_A(f) \in |k^\times|$.
\end{lemma}

\begin{proof}
There exists a tuple $r=(r_1,\ldots,r_n)$ of positive real numbers linearly independent over $|k^\times|$ such that $A \widehat{\otimes}_{k}K_r$ is a strictly $K_r$-affinoid algebra.
Since the homomorphism $A \to A \widehat{\otimes}_{k}K_r$ is isometry with respect to spectral radii by \cref{A-to-AK},  we may assume that $A$ is a strictly $K_r$-affinoid algebra.
By the Noether normalization (\cite{BGR}*{Corollary~6.1.2.2}), there exists a finite and admissible  injection $\varphi \colon \mathcal{T}:= K_r\{T_1,\ldots,T_d\} \to A$.
By \cite{Temkin04}*{Proposition~3.1~(iii)}, the reduction map $\wt{\mathcal{T}} \to \wt{A}$ is finite.
We take a homogeneous generator $h_1,\ldots,h_m \in \wt{A}\backslash \{0\}$ and set $s_i:=\rho_A(f_i)$.
Then for every $g \in A$, there exists $g_1,\ldots,g_m \in \mathcal{T}$ such that $\wt{g}=\sum_{i=1}^m \wt{g_i}h_i$, and in particular, we have $\rho(g)=\sum_{i=1}^m \rho(g_i)s_i$.
In conclusion, we obtain $\rho(A)=\sum_{i=1}^m \rho(\mathcal{T})s_i$.
On the other hand, since $s_i \in ||\mathcal{T}^\times||^{\Q}=|K_r^\times|^\Q$, there exists a positive integer $l$ such that $s_i^l \in |K_r^\times|^\Q$, and in particular, we have $\rho(A) \backslash \{0\} \subseteq |k^\times|<r_1^{1/l},\ldots,r_n^{1/l}>$.
We set $r:=\rho(f)$ and write $e=\lambda r_1^{i_1/l}\cdots r_n^{i_n/l}$ for some $i_1,\ldots,i_n \in \Z_{\geq 0}$ and $\lambda \in |k^\times|$.
Since $(r_1,\ldots,r_n)$ is linearly independent over $|k^\times|$, we obtain
\[
\rho(f^{1/p^e})=r^{1/p^e}=\lambda^{1/p^e}r_1^{i_1/p^el}\cdots r_n^{i_n/p^el} \in \rho(A).
\]
Therefore, we have $i_1=\cdots=i_n=0$ by $\rho(A) \backslash \{0\} \subseteq |k^\times|<r_1^{1/l},\ldots,r_n^{1/l}>$, and in particular, it follows that $\rho(f)=\lambda \in |k^\times|$.
\end{proof}

\begin{lemma}\label{f-in-k}
We use \cref{notation-non-triv}.
We assume $k$ is an algebraically closed field.
Let $r=(r_1,\ldots,r_m)$ be a tuple of positive real numbers linearly independent over $|k^\times|$.
Let $K_r \to L$ be a finite extension of fields and $f \in L$ with $|f|_L \in |k^\times|$.
If $f-a$ admits a compatible system of $p$-power roots for every $a \in k$ with $|a| \geq |f|_L$, then $f \in k$.    
\end{lemma}

\begin{proof}
Let  
\[
P(X) = X^n + f_1 X^{n-1} + \cdots + f_n \in K_r[X]
\]  
be the minimal polynomial of \( f \). Define  
\[
P_a(X) := P(X + a) = X^n + f_{1,a} X^{n-1} + \cdots + f_{n,a} \in K_r[X],
\]  
so that \( f_{n,a} = a^n + f_1 a^{n-1} + \cdots + f_n = P(a) \).  
Since \( P_a(f - a) = P(f) = 0 \) and \( a \in k \), the polynomial \( P_a(X) \) is the minimal polynomial of \( f - a \).
Write \( f_i = \sum a_{i,\nu} T^\nu \in K_r \) with \( a_{i,\nu} \in k \).  
Let \( a \in k \) be a root of the polynomial \( X^n + a_{1,0} X^{n-1} + \cdots + a_{n,0} \) with maximal value.  
Then we have:
\begin{equation} \label{eq:a-f}
|a| \overset{(\star_1)}{\geq} |a_{n,0}|^{1/n} \eqtag{2} |f_n|_{K_r}^{1/n} \eqtag{3} |f|_L,
\end{equation}
where \( (\star_1) \) follows from the fact that \( a_{n,0} \) is the product of the roots of the polynomial \( X^n + a_{1,0} X^{n-1} + \cdots + a_{n,0} \), and \( (\star_3) \) holds because \( |f|_L = |\mathrm{Nm}(f)|_{K_r}^{1/n} = |f_{n,0}|^{1/n} \).
We now verify equation \( (\star_2) \).  
Since \( |f|_L \in |k^\times| \) and by \( (\star_3) \), we have \( |f_n|_{K_r} \in |k^\times| \).  
Because \( (r_1, \ldots, r_n) \) are independent over \( |k^\times| \), we conclude that \( |f_n| = |a_{n,0}| \), and so equation \( (\star_2) \) holds.
Therefore, by assumption, \( f - a \) admits a compatible system of \( p \)-power roots, and in particular, \( |f - a|_L \in |k^\times| \) by \cref{norm-f-in}.  
Moreover, we compute:
\begin{align*}
    |f - a|_L 
    &= |f_{n,a}|_{K_r}^{1/n} \\
    &= \max \left\{ \left| a^n + a_{1,0} a^{n-1} + \cdots + a_{n,0} \right|, \ \max_{\nu \neq 0} \left| a_{1,\nu} a^{n-1} + \cdots + a_{n,\nu} \right| r^\nu \right\}^{1/n}.
\end{align*}
Since \( a^n + a_{1,0} a^{n-1} + \cdots + a_{n,0} = 0 \), we see that \( |f - a|_L \notin |k^\times| \) if \( f \neq a \).  
In conclusion, this forces \( f = a \in k \), as desired.
\end{proof}

\begin{lemma}\label{reduce-to-alg-cl}
We use \cref{notation-non-triv}.
Let $r:=(r_1,\ldots,r_n)$ be a tuple of positive real numbers and 
\[
f \in k\{r^{-1}T\}=k\{r_1^{-1}T_1,\ldots,r_n^{-1}T_n\}.
\]
For $a \in \bar{k}$ with $|a|_{\bar{k}} > \|f\|_{k\{r^{-1}T\}}$, the element $f + a \in k\{r^{-1}T\} \otimes_k \bar{k}$ admits a compatible system of $p$-power roots.
\end{lemma}

\begin{proof}
Fix \( a \in \bar{k} \) such that \( |a|_{\bar{k}} > \|f\|_{k\{r^{-1}T\}} \), and choose a compatible system of \( p \)-power roots \( \{a^{1/p^e} \in \bar{k}\} \) of \( a \).  
Let \( k_1 := k(a^{1/p}) \), equipped with the unique extension of the valuation \( |\cdot|_k \).  
Since the map \( k\{r^{-1}T\} \to k_1\{r^{-1}T\} \) is an isometry and the norm \( \|\cdot\|_{k_1\{r^{-1}T\}} \) is multiplicative, there exists a \( p \)-th root \( (f + a)^{1/p} \) of \( f + a \) such that  
\[
\|(f + a)^{1/p} - a^{1/p}\| < \|(f + a)^{1/p}\| = \|a^{1/p}\|
\]  
by \cref{exists-p-th-root}.
Set \( f_1 := (f + a)^{1/p} - a^{1/p} \).  
Then \( \|f_1\| < \|(f + a)^{1/p}\| = \|a^{1/p}\| \).  
By replacing \( f \) with \( f_1 \), \( a \) with \( a^{1/p} \), and \( k_1 \) with \( k_1(a^{1/p^2}) \), we may repeat the above argument to obtain a \( p^2 \)-th root of \( f + a \) in \( k_2\{r^{-1}T\} \). 
Iterating this process, we obtain a compatible system of \( p \)-power roots of \( f + a \) in \(k\{r^{-1}T\} \otimes_k \bar{k}\), as desired.
\end{proof}

\begin{lemma}\label{essential-lemma}
We use \cref{notation-non-triv}.
We assume $k$ is an algebraically closed field.
Let $r=(r_1,\ldots,r_n)$ be a tuple of positive real integer independent over $|k^\times|^\Q$.
Let $A$ be a reduced strictly $K_r$-affinoid algebra such that $\rho:=\rho_A$ is multiplicative.
Let $\{f_n \in A\}_{n \geq 1}$ and $\{v_n \in |k^\times|\}_{n \geq 1}$ such that $f:=\lim_{n \to \infty} f_n$ exists and $f_n+a$ admits a compatible system $p$-power roots for $a \in k$ with $|a| \geq v_n$.
If $f \neq 0$, then there exists $n_0 \geq 1$ such that $f_n,f \in k$ for $n \geq n_0$.
\end{lemma}

\begin{proof}
Suppose that $f \neq 0$, and set $r := \rho(f) > 0$. Then there exists an integer $n_1 \geq 1$ such that $\rho(f - f_n) < r$ and $v_n < r$ for all $n \geq n_1$. In particular, we have $\rho(f_n) = r$ for all $n \geq n_1$.

We first prove that $f_n$ admits a compatible system of $p$-power roots for all $n \geq n_1$. 
Fix $n \geq n_1$, and take $a \in k$ such that $|a| = v_n$. 
If $f_n \in k$, the assertion is clear.
Thus, we may assume $f_n \notin k$.
Then $f_n + a$ admits a compatible system of $p$-power roots, and in particular, $f_n + a \in A^\times$ by \cref{p-power-root-to-unit}. Moreover, since
\[
\rho(f_n - (f_n + a)) = |a| = v_n < r = \rho(f_n),
\]
it follows from \cref{norm-f-in} that $f_n$ also admits a compatible system of $p$-power roots.

Next, we prove that $f + a$ admits a compatible system of $p$-power roots for every $a \in k$. If $f \in k$, the assertion is clear. Thus, we may assume $f \notin k$. Fix $a \in k$. Then $\rho(f + a) > 0$. Take $n \geq n_1$ such that $|a| > v_n$ or $a = 0$, and $\rho(f - f_n) < \rho(f + a)$. Then we have
\[
\rho(f + a - (f_n + a)) = \rho(f - f_n) < \rho(f + a).
\]
Since $f_n + a$ admits a compatible system of $p$-power roots, it follows from \cref{exists-p-th-root} that $f + a$ also admits such a system. In particular, $f \in A^\times$ by \cref{p-power-root-to-unit}.

Now take a maximal ideal $\mathfrak{m} \subseteq A$, and let $\bar{f}, \bar{f}_n \in L := A/\mathfrak{m}$ denote the images of $f, f_n$ in $L$ for each $n \geq 1$. Since $A$ is a strictly $K_r$-affinoid algebra, $L$ is a finite extension of $K_r$. 
Set $r' := |\bar{f}|_L>0$. Then there exists $n_0 \geq n_1$ such that $v_n < r'$ and $|\bar{f}_n - \bar{f}| < r'$ for all $n \geq n_0$. Note that $|\bar{f}_n| = r'$. 
Fix $n \geq n_0$, and let $a \in k$ with $|a| \geq r' > v_n$ or $a = 0$. Then $f_n - a$ admits a compatible system of $p$-power roots, and hence so does $\bar{f}_n - a$ in $L$. By \cref{norm-f-in,f-in-k}, we conclude that $\bar{f}_n \in k$, and in particular, there exists $a_n \in k$ such that $f - a_n \in \mathfrak{m}$. Furthermore, $|a_n| = |\bar{f}_n|_L = r' > v_n$, so $f_n = a_n$ by \cref{p-power-root-to-unit}. 
On the other hand, since $f - a$ admits a compatible system of $p$-power roots for every $a \in k$, it follows by the same argument that $f \in k$, as desired.
\end{proof}

\begin{theorem}\label{reduced-case}\textup{(cf.~\cref{main-thm})}
We use \cref{notation-non-triv}.
Let $\varphi \colon B \to A$ be a $k$-algebra homomorphism between $k$-affinoid algebras such that $A$ is reduced.
Then $\varphi$ is bounded.    
\end{theorem}

\begin{proof}
By \cite{Temkin10}*{Fact~3.1.2.1~(iii)}, we may assume that the norm on $A$ coincides with $\rho_A$.
Take an admissible surjection $k\{s^{-1}S\} \to B$ for some $s \in \mathbb{R}_{> 0}^m$. 
If the composition $k\{s^{-1}S\} \to B \to A$ is bounded, then so is $\varphi$, and we may assume $B = k\{s^{-1}S\}$ without loss of generality.
Let $r = (r_1,\ldots,r_q)$ be a tuple of positive real numbers that are $\mathbb{Q}$-linearly independent over $|k^\times|^\mathbb{Q}$ such that $A \widehat{\otimes}_k K_r$ is a strictly $K_r$-affinoid algebra.
Fix a sequence $\{g_n \in B\}_{n \geq 1}$ such that $\lim_{n \to \infty} g_n = 0$ and that $f := \lim_{n \to \infty} \varphi(g_n)$ exists. By the closed graph theorem (\cite{BGR}*{Section~2.8.1}), it suffices to show that $f = 0$.
By \cref{A-to-AK}, the spectral radius of $A \widehat{\otimes}_k K_r$ is a norm, thus $A \widehat{\otimes}_k K_r$ is reduced.
Since the map $A \to A \widehat{\otimes}_k K_r$ is injective by \cite{Gruson}*{Section~3,~Theorem~1(4)}, we may assume that $A$ is a strictly $K_r$-affinoid algebra.
Set $f_n := \varphi(g_n)$, and take a sequence $\{v_n \in |k^\times|\}_{n \geq 1}$ such that $\|g_n\| < v_n$ and $\lim_{n \to \infty} v_n = 0$.
By \cref{reduce-to-alg-cl}, for every $a \in \bar{k}$ and $n \geq 1$ such that $|a|_{\bar{k}} \geq v_n$, the element $g_n + a$ admits a compatible system of $p$-power roots in $B \otimes_k \bar{k}$. Since $\varphi$ is a $k$-algebra homomorphism, the same holds for $f_n + a$ in $A \otimes_k \bar{k}$.
In particular, $f_n + a$ admits such a system in $A' := A \widehat{\otimes}_k \bar{k}$. Note that $A'$ is a strictly $K_r' := K_r \widehat{\otimes}_k \bar{k}$-affinoid algebra, and we have
\[
K_r' \simeq \bar{k}\{r^{-1}T, rT^{-1}\}.
\]
Suppose for contradiction that $f \neq 0$.
For a point $x$ in the Shilov boundary of $A'$, there exists a bounded $K_r'$-algebra homomorphism $\psi_x \colon A' \to A'_x$ to a strictly $K_r'$-affinoid algebra $A'_x$ such that $A'_x$ is reduced, $\rho_{A'_x}$ is multiplicative, and $\psi_x^*(\rho_{A'_x}) = x$.
Then there exists $x \in \Gamma(A')$ such that $\psi_x(f) \neq 0$ by \cref{Shilov-boundary}. Then $\psi_x(f)$, $\{\psi_x(f_n)\}$, and $\{v_n\}$ satisfy the assumptions of \cref{essential-lemma}, so there exists an integer $n_x \geq 1$ such that $\psi_x(f), \psi_x(f_n) \in \bar{k}$ for all $n \geq n_x$.
Set $a := \psi_x(f)$ and $a_n := \psi_x(f_n)$ for $n \geq n_x$.
We claim the following:
\begin{claim}\label{cl1:reduced-case}
The field extension $k \subseteq k' := k(a, a_n \mid n \geq n_x)$ is finite.
\end{claim}

\begin{proof}[Proof of \cref{cl1:reduced-case}]
Take a maximal ideal $\mathfrak{m} \subseteq A$ and set $L := A / \mathfrak{m}$. Then $k \subseteq k'$ is a subextension of $k \subseteq L$. Since $K_r \subseteq L$ is finite, there exist finitely many elements $b_1, \ldots, b_l \in \{a, a_n \mid n \geq n_x\}$ such that $k' \subseteq K_r(b_1, \ldots, b_l)$.
Then for each $n \geq n_x$, there exists a polynomial $G_n(X_1,\ldots,X_l) \in K_r[X_1, \ldots, X_l]$ such that
\[
a_n = G_n(b_1, \ldots, b_l) \in K_r(b_1, \ldots, b_l) \simeq k(b_1, \ldots, b_l)\{r^{-1}T, rT^{-1}\}.
\]
Comparing the degree zero parts with respect to $T$, we see that $a_n \in k(b_1, \ldots, b_l)$. By the same reasoning, we also obtain $a \in k(b_1, \ldots, b_l)$.
Hence, $k \subseteq k' = k(b_1, \ldots, b_l)$ is a finite extension, as claimed.
\end{proof}
Now consider the $k'$-module homomorphism
\[
\varphi_x \colon B \otimes_k k' \xrightarrow{\varphi_{k'}} A \otimes_k k' \to A'_x.
\]
Since $a_n = \psi_x(f_n)$ and $a = \psi_x(f)$, we have $g_n - a_n \in \ker(\varphi_x)$ for all $n \geq n_x$.
As $k'$ is finite over $k$ by \cref{cl1:reduced-case}, $B \otimes_k k'$ is a strictly $k$-affinoid algebra, and in particular, $\ker(\varphi_x)$ is closed.
Hence, we obtain
\[
\lim_{n \to \infty} (g_n - a_n) = -a \in \ker(\varphi_x).
\]
Thus, $\psi_x(f) = a = 0$, contradicting the choice of $x$ such that $\psi_x(f) \neq 0$.
We conclude that $f = 0$, as desired.
\end{proof}

\subsection{The general case}
In this subsection, we apply the result of Subsection~2.2 (\cref{reduced-case}) to interpret \cref{main-thm} as a kind of deformation problem. This perspective enables us to relate the boundedness of $k$-algebra homomorphisms to derivations over a polynomial rings, thereby connecting it to the algebraic properties of the base field $k$.

\begin{lemma}\label{deformation-prob}
Let $k$ be a complete non-archimedean valued field, and let $r = (r_1, \ldots, r_n) \in \mathbb{R}_{>0}^n$.
Let $A$ be a reduced $k$-affinoid algebra, and let $\psi \colon k\{r^{-1}T\} \to A$ be a bounded $k$-algebra homomorphism.
Then the following conditions are equivalent:
\begin{enumerate}
    \item For any $k$-algebra homomorphism $\varphi \colon k\{r^{-1}T\} \to A'$ to a $k$-affinoid algebra $A'$ such that there exists a bounded isomorphism $(A')^{\mathrm{red}} \xrightarrow{\sim} A$ making the following diagram commute:
    \[
    \psi = \left( k\{r^{-1}T\} \xrightarrow{\varphi} A' \twoheadrightarrow (A')^{\mathrm{red}} \xrightarrow{\sim} A \right),
    \]
    the map $\varphi$ is bounded.
    
    \item For every ideal $J \subseteq A$, we have
    \[
    \Hom_{k\{r^{-1}T\}}\left(\Omega_{k\{r^{-1}T\}/k[T]}, A/J\right) = 0,
    \]
    where $A/J$ is regarded as a $k\{r^{-1}T\}$-algebra via the composition
    \[
    k\{r^{-1}T\} \xrightarrow{\psi} A \to A/J.
    \]
\end{enumerate}
\end{lemma}

\begin{proof}
We first prove the implication (1) $\Rightarrow$ (2) by contrapositive.
Assume there exists an ideal $J \subseteq A$ and a non-zero derivation
\[
0 \neq \delta \in \mathrm{Der}_{k[T]}\left(k\{r^{-1}T\}, A/J\right) = \Hom_{k\{r^{-1}T\}}\left(\Omega_{k\{r^{-1}T\}/k[T]}, A/J\right).
\]
Define a seminormed $A$-module $(A',\|\cdot\|_{A'}) := (A, \rho_A) \oplus^{\max} (A/J, \|\cdot\|_q)$, where $\|\cdot\|_q$ is the quotient seminorm induced by the natural surjection $\pi \colon A \to A/J$. We endow $A'$ with the following ring structure:
\[
(a,b) \cdot (a',b') := (aa', \pi(a)b' + \pi(a')b).
\]

\begin{claim}\label{cl:const-sn-ring}
The pair $(A',\|\cdot\|_{A'})$ satisfies the following properties:
\begin{enumerate}
    \item $(A',\|\cdot\|_{A'})$ is a seminormed ring.
    \item The map $(A, \rho_A) \to (A', \|\cdot\|_{A'})$ is a seminormed finite ring homomorphism.
    \item The induced map $A \to (A')^{\mathrm{red}}$ is a bounded isomorphism.
\end{enumerate}
\end{claim}

\begin{proof}[Proof of \cref{cl:const-sn-ring}]
(1) follows from the computation:
\begin{align*}
\|(a,b)\cdot (a',b')\|
&= \max\{\rho_A(aa'), \|\pi(a)b' + \pi(a')b\|_q\} \\
&\leq \max\{\rho_A(a)\rho_A(a'), \rho_A(a)\|b'\|_q, \rho_A(a')\|b\|_q\} \\
&\leq \max\{\rho_A(a), \|b\|_q\} \cdot \max\{\rho_A(a'), \|b'\|_q\} \\
&= \|(a,b)\| \cdot \|(a',b')\|.
\end{align*}
(2) follows from the fact that the $A$-module map
\[
A \oplus^{\max} A \to A \oplus^{\max} A/J,\quad (a,b) \mapsto (a,\pi(b))
\]
is an admissible surjection.
Finally, we prove assertion (3).
The inclusion $ p \colon A \to A'$ given by $a \mapsto (a,0)$ induces a ring homomorphism $p^{\mathrm{red}} \colon A \to (A')^{\mathrm{red}}$.
This is an isomorphism since $A$ is reduced and $(0,b)^2 = 0$ for all $b \in A/J$.
Moreover, since $p$ is an isometry, $p^{\mathrm{red}}$ is a bounded isomorphism.
\end{proof}

Now define a map
\[
\varphi \colon k\{r^{-1}T\} \to A',\quad f \mapsto (\psi(f), \delta(f)),
\]
then $\varphi$ is $k$-algebra homomorphism since $\delta$ is a derivation over $k$.
We claim that $\varphi$ is not bounded.
To see this, define the bounded homomorphism
\[
\varphi' \colon k\{r^{-1}T\} \xrightarrow{\psi} A \xrightarrow{p} A'.
\]
Then $\varphi|_{k[T]} = \varphi'|_{k[T]}$ but $\varphi \neq \varphi'$ since $\delta$ is a non-zero derivation over $k[T]$.
Since $k[T]$ is dense in $k\{r^{-1}T\}$, we conclude that $\varphi$ is not bounded.

Next, we prove the implication (2) $\Rightarrow$ (1).
Let $\pi_0 \colon A' \twoheadrightarrow (A')^{\mathrm{red}} \xrightarrow{\sim} A$ be the given bounded $k$-algebra homomorphism.
Let $\varphi \colon k\{r^{-1}T\} \to A'$ be a $k$-algebra homomorphism such that $\pi_0 \circ \varphi = \psi$.
For $1 \leq i \leq n$, set $f_i := \varphi(T_i) \in A'$.
Since $\rho_{A'}(f_i) = \rho_A(\psi(T_i)) \leq r_i$, there exists a bounded $k$-algebra homomorphism
\[
\varphi' \colon k\{r^{-1}T\} \to A'
\]
with $\varphi'(T_i) = f_i$ for all $i$ (by \cite{Temkin10}*{Exercise~3.1.2.5}).
Then $\varphi|_{k[T]} = \varphi'|_{k[T]}$ and thus
\[
\psi|_{k[T]} = \pi_0 \circ \varphi|_{k[T]} = \pi_0 \circ \varphi'|_{k[T]}.
\]
Since both $\psi$ and $\pi_0 \circ \varphi'$ are bounded and $k[T]$ is dense in $k\{r^{-1}T\}$, we get $\psi = \pi_0 \circ \varphi'$.

\begin{claim}\label{cl:deformation}
There exists a filtration of ideals of $A'$
\[
0 = I_0 \subseteq I_1 \subseteq \cdots \subseteq I_l = \sqrt{(0)_{A'}}
\]
such that for each $0 \leq m < l$, there exists an ideal $J_m \subsetneq A$ with $I_{m+1}/I_m \simeq A/J_m$.
\end{claim}

\begin{proof}[Proof of \cref{cl:deformation}]
Let $I := \sqrt{(0)_{A'}}$, so $I^s = 0$ for some integer $s \geq 1$.
Each $I^j/I^{j+1}$ is a finite $A'/I \simeq A$-module.
Then for each $j$, there is a filtration
\[
0 = M_0^{(j)} \subseteq \cdots \subseteq M_{t_j}^{(j)} = I^j/I^{j+1}
\]
with $M_{i+1}^{(j)}/M_i^{(j)} \simeq A/J_i^{(j)}$ for some ideals $J_i^{(j)} \subseteq A$.
Lift each $M_i^{(j)}$ to an ideal $I_i^{(j)} \subseteq A'$ such that $I^{j+1} \subseteq I_i^{(j)} \subseteq I^j$ and $I_i^{(j)}/I^{j+1} = M_i^{(j)}$.
Then the sequence
\[
0=I^{(s-1)}_0 \subseteq I^{(s-1)}_1 \subseteq \cdots \subseteq I^{(s-1)}_{t_{s}}=I^{s-1} \subseteq I^{(s-2)}_1 \subseteq \cdots \subseteq I^{(0)}_{t_0-1} \subseteq  I^{(0)}_{t_0}=I
\]
is the desired sequence.
\end{proof}

Apply \cref{cl:deformation} to get $I_0, \ldots, I_l$ and $J_0, \ldots, J_{l-1}$.
Set $\pi_m \colon A' \to A'_m := A'/I_m$ and define $\varphi_m := \pi_m \circ \varphi$, $\varphi'_m := \pi_m \circ \varphi'$.
We claim $\varphi_m = \varphi'_m$ for all $m$, by downward induction on $m$.
For $m = l$, $\varphi_l = \varphi'_l = \psi$.
Assume $\varphi_{m+1} = \varphi'_{m+1}$.
Then $\varphi_m - \varphi'_m$ induces a derivation $\delta \colon k\{r^{-1}T\} \to A/J_m$ by the exact sequence
\[
0 \to A/J_m \to A_m \to A_{m+1} \to 0.
\]
Since $\varphi_m$ and $\varphi'_m$ agree on $k[T]$, $\delta$ is a derivation over $k[T]$.
By assumption (2), $\delta = 0$, hence $\varphi_m = \varphi'_m$.
In conclusion, $\varphi = \varphi_0 = \varphi'_0 = \varphi'$, and thus $\varphi$ is bounded.
\end{proof}

\begin{lemma}\label{computation-derivation}
Let $k$ be a complete non-archimedean valued field of characteristic $p \geq 0$.
Let $r \in (0,1) \setminus |k^\times|^\mathbb{Q}$. Then if either $p=0$, or $p>0$ and $[k:k^p] = \infty$, we have
\[
\Hom_{k\{r^{-1}T\}}\left(\Omega_{k\{r^{-1}T\}/k[T]},K_r\right) \neq 0,
\]
$k^p:=\{a^p \mid a\in k\}$.
\end{lemma}

\begin{proof}
We first consider the case where $p=0$.
Define a sequence $\{i_m\}_{m \geq 1} \subset \mathbb{Z}_{>0}$ by $i_1=2$ and $i_{m+1}=m(1+i_m)+1$.
Define
\[
f := \sum a_i T^i \in k\{r^{-1}T\}, \quad \text{where } a_i =
\begin{cases}
    1 & \text{if } i = i_m \text{ for some } m \geq 1, \\
    0 & \text{otherwise}.
\end{cases}
\]

\begin{claim}\label{cl:omega-charazero}
The element $f$ is not integral over $k(T)$.
\end{claim}

\begin{proof}[Proof of \cref{cl:omega-charazero}]
Suppose for contradiction that $f$ is integral over $k(T)$, i.e., there exists a monic polynomial 
\[
P(X) = X^n + h_1' X^{n-1} + \cdots + h_n' \in k(T)[X]
\]
such that $P(f) = 0$.
Then there exists $h_0' \in k[T] \setminus \{0\}$ such that $h_i := h_i' h_0' \in k[T]$ for all $i$ and
\begin{equation}\label{eq:int-f}
  h_0 f^n + h_1 f^{n-1} + \cdots + h_n = 0.
\end{equation}
Let $m := \max\{n, \deg h_0, \ldots, \deg h_n\}$, and decompose $f$ as $f = f_1 + f_2$ where
\[
f_1 := \sum_{i \leq i_m} a_i T^i, \qquad f_2 := \sum_{i > i_m} a_i T^i.
\]
Substituting the decomposition $f = f_1 + f_2$ into \eqref{eq:int-f}, we obtain
\begin{equation}\label{eq:int-f-2}
    h_0 f_1^n + h_1 f_1^{n-1} + \cdots + h_n + f_2 h = 0
\end{equation}
for some $h \in k\{r^{-1}T\}$.
Define the set $I(g)$ for $g = \sum b_i T^i \in k\{r^{-1}T\}$ by
\[
I(g) := \{i \in \mathbb{Z}_{\geq 0} \mid b_i \neq 0\}.
\]
Then $I(f_2 h) \subseteq \mathbb{Z}_{\geq i_{m+1}}$ and
\[
\max I(h_0 f_1^n + \cdots + h_n) \leq m + n i_m < i_{m+1},
\]
so they are disjoint, and hence we have
\[
h_0 f_1^n + \cdots + h_n = 0
\]
by \eqref{eq:int-f-2}.
However, note that
\[
\deg(h_0 f_1^n) = \deg(h_0) + n \deg(f_1) = \deg(h_0) + n i_m \geq  n i_m,
\]
while the degrees of the other terms satisfy
\[
 \deg(h_1 f_1^{n-1} + \cdots + h_n) \leq m + (n-1)i_m < n i_m.
\]
This contradicts \eqref{eq:int-f-2}. Hence, $f$ is not integral over $k(T)$.
\end{proof}

By \cref{cl:omega-charazero} and \cite{Matsumura}*{Theorem~26.5}, the differential $df \in \Omega_{K/k(T)}$ is non-zero, where $K := \mathrm{Frac}(k\{r^{-1}T\})$.
Hence, there exists a $K$-linear map $\delta' \colon \Omega_{K/k(T)} \to K_r$ such that $\delta'(df) = 1$.
Composing with the natural map yields
\[
\delta \colon \Omega_{k\{r^{-1}T\}/k[T]} \to \Omega_{K/k(T)} \xrightarrow{\delta'} K_r,
\]
which satisfies $\delta(df) = 1$.
Thus, we obtain
\[
\Hom_{k\{r^{-1}T\}}\left(\Omega_{k\{r^{-1}T\}/k[T]}, K_r\right) \neq 0.
\]

Now suppose $p > 0$ and $[k:k^p] = \infty$.
Let $\{x_\lambda\}_{\lambda \in \Lambda}$ be a $p$-basis of $k$ over $\mathbb{F}_p$.
By multiplying each $x_\lambda$ by a suitable non-zero element of $k^p$, we may assume $|x_\lambda| \leq 1$ for all $\lambda$.
Since $[k:k^p] = \infty$, the set $\Lambda$ is infinite.
Choose distinct elements $\{\lambda_n\}_{n \geq 0} \subset \Lambda$ and define
\[
f := \sum_{i \geq 0} x_{\lambda_i} T^i \in k\{r^{-1}T\}.
\]
Arguing as in the characteristic zero case, it suffices to show $f \notin K^p(k(T))$ since $df \neq 0$ in $\Omega_{K/k(T)}$ by \cite{Matsumura}*{Theorem~26.5}.
Assume for contradiction that
\[
f = g_1' f_1'^p + \cdots + g_n' f_n'^p
\]
for some $g_i' \in k(T)$ and $f_i' \in K$.
Then choose $g_0 \in k[T] \setminus \{0\}$ and $f_0 \in k\{r^{-1}T\} \setminus \{0\}$ such that
\[
g_i := g_0 g_i' \in k[T], \qquad f_i := f_0 f_i' \in k\{r^{-1}T\}.
\]
Then
\begin{equation}\label{eq:ch-p}
f_0^p g_0 f = g_1 f_1^p + \cdots + g_n f_n^p.
\end{equation}

Write $f_i = \sum_\nu a_{i,\nu} T^\nu$, $g_i = \sum_\nu b_{i,\nu} T^\nu$, and expand each $b_{i,\nu} \in k$ as
\[
b_{i,\nu} = \sum_{\mu \in M} b_{i,\nu,\mu}^p x^\mu,
\]
where:
\begin{itemize}
  \item $M := \{\mu \colon \Lambda \to \{0,\ldots,p-1\} \mid \mu(\lambda) = 0 \text{ for all but finitely many } \lambda\}$,
  \item $x^\mu := \prod_{\lambda \in \Lambda} x_\lambda^{\mu(\lambda)}$,
  \item each $b_{i,\nu,\mu} \in k$ is zero for all but finitely many $\mu$.
\end{itemize}
Define
\[
\Lambda' := \left\{ \lambda \in \Lambda \,\middle|\, \exists\, i, \nu, \mu \text{ with } \mu(\lambda) \neq 0 \text{ and } b_{i,\nu,\mu} \neq 0 \right\}.
\]
Then $\Lambda'$ is finite, so there exists $j \geq 0$ such that $\lambda_j \notin \Lambda'$.
Choose $\nu_0 \geq 0$ such that
\[
\sum_{\nu_0 = p\nu_1 + \nu_2} a_{0,\nu_1}^p b_{0,\nu_2} \neq 0,
\]
and pick $\mu_0 \in M$ such that
\[
\sum_{\nu_0 = p\nu_1 + \nu_2} a_{0,\nu_1}^p b_{0,\nu_2,\mu_0}^p \neq 0.
\]
The coefficient of $T^{\nu_0 + j}$ in $f_0^p g_0 f$ is
\[
\alpha:=\sum_{\nu_0 + j = p\nu_1 + \nu_2 + \nu_3} a_{0,\nu_1}^p b_{0,\nu_2} x_{\lambda_{\nu_3}},
\]
and the coefficient of $x^{\mu_0}x_{\lambda_j}$ in the basis expansion of $\alpha$ with respect to $\{x^\mu \mid \mu \in M\}$ is 
\[
\sum_{\nu_0=\nu_1+\nu_2} a_{0,\nu_1}^p b_{0,\nu_2,\mu_0}^p \neq 0.
\]
by $\lambda_j \notin \Lambda'$.
On the other hand, the coefficient of $T^{\nu_0 + j}$ in the right-hand side of \eqref{eq:ch-p} lies in the $k^p$-span of $\{x^\mu \mid \mu \in M'\}$, where 
\[
M' := \{\mu \colon \Lambda' \to \{0,\ldots,p-1\} \mid \mu(\lambda) = 0 \text{ for all but finitely many } \lambda\},
\]
and its $x^{\mu_0} x_{\lambda_j}$-coefficient must vanish since $\lambda_j \notin \Lambda'$.
This contradicts \eqref{eq:ch-p}, so $f \notin K^p(k(T))$, completing the proof.
\end{proof}

\begin{lemma}\label{F-finite-case}
Let $k$ be a complete non-archimedean valued field of characteristic $p > 0$.
Let $r = (r_1,\ldots,r_n)$ be a tuple of positive real numbers.
If $[k : k^p] < \infty$, then $\Omega_{k\{r^{-1}T\}/k[T]} = 0$.
\end{lemma}

\begin{proof}
Let $x_1,\ldots,x_s$ be a basis of $k$ over $k^p$.
We equip $k^p$ with the subspace norm induced from $k$. Then $k^p$ is complete, and the map
\[
\bigoplus^s k^p \longrightarrow k,\quad (a_1,\ldots,a_s) \mapsto \sum_{i=1}^s a_i^p x_i
\]
is a bounded isomorphism of normed $k^p$-vector spaces \cite{Schneider2009}*{Proposition~1.1.6}.
In particular, there exists a constant $C > 0$ such that for any $a = \sum_{i=1}^s a_i^p x_i \in k$ with $a_i \in k$, we have
\begin{equation}\label{eq:F-fin}
    C^{-1} |a| \leq \max\{|a_1^p|, \ldots, |a_s^p|\} \leq C |a|.
\end{equation}
Now let $f = \sum_{\nu \in \mathbb{Z}_{\geq 0}^n} a_\nu T^\nu \in k\{r^{-1}T\}$. For each $e \in \{0, \ldots, p-1\}^n$, define:
\begin{itemize}
    \item $f_e := \sum_{\nu' \in \mathbb{Z}_{\geq 0}^n} a_{p\nu' + e} T^{p\nu' + e}$,
    \item express each coefficient $a_{p\nu'+e} \in k$ as
    \[
    a_{p\nu'+e} = \sum_{i=1}^s a_{p\nu'+e,i}^p x_i, \quad \text{with } a_{p\nu'+e,i} \in k,
    \]
    using the $k^p$-basis $\{x_1, \ldots, x_s\}$,
    \item set $f_{e,i} := \sum_{\nu' \in \mathbb{Z}_{\geq 0}^n} a_{p\nu'+e,i} T^{\nu'} \in k[[T]]$.
\end{itemize}
We claim that each $f_{e,i}$ lies in $k\{r^{-1}T\}$. Indeed, by inequality \eqref{eq:F-fin}, we have
\[
|a_{p\nu'+e,i}| \cdot r^{\nu'} 
\leq C^{1/p} |a_{p\nu'+e}|^{1/p} r^{\nu'} 
= C^{1/p} \left( |a_{p\nu'+e}| \cdot r^{p\nu'+e} \right)^{1/p} \cdot r^{-e/p} \longrightarrow 0 \quad (\nu' \to \infty),
\]
since $f \in k\{r^{-1}T\}$.
Hence, we have
\[
f = \sum_{i=1}^s \sum_{e \in \{0,\ldots,p-1\}^n} f_{e,i}^p \cdot x_i T^e.
\]
We now compute the differential:
\[
df = \sum_{i=1}^s \sum_{e \in \{0,\ldots,p-1\}^n} f_{e,i}^p \cdot d(x_i T^e),
\]
which shows that $df$ lies in the image of the natural map
\[
\Omega_{k[T]/k} \otimes_{k[T]} k\{r^{-1}T\} \longrightarrow \Omega_{k\{r^{-1}T\}/k}.
\]
Therefore, the cokernel $\Omega_{k\{r^{-1}T\}/k[T]}$ vanishes, completing the proof.
\end{proof}

\begin{proof}[Proof of \cref{main-thm}]
The first assertion follows from \cref{reduced-case}.

We prove the implication (1) $\Rightarrow$ (2) by contraposition.
Assume that $|k^\times|^{\mathbb{Q}} \neq \mathbb{R}_{>0}$ and choose $r \in (0,1) \setminus |k^\times|^{\mathbb{Q}}$.
Suppose that either $p = 0$ or $[k : k^p] = \infty$.
Then by \cref{computation-derivation}, we have
\[
\Hom_{k\{r^{-1}T\}}\left( \Omega_{k\{r^{-1}T\}/k[T]}, K_r \right) \neq 0.
\]
By \cref{deformation-prob}, there exists a $k$-algebra homomorphism
\[
\varphi' \colon k\{r^{-1}T\} \to A'
\]
to a $k$-affinoid algebra $A'$ such that $(A')^{\mathrm{red}} \simeq K_r$, and such that $\varphi'$ is not bounded. This contradicts condition (1). Therefore, (1) implies (2).

We now prove the implication (2) $\Rightarrow$ (1).
If $|k^\times|^{\mathbb{Q}} = \mathbb{R}_{>0}$, then the claim follows from the strict case \cite{Berkovich_note}*{Proposition~2.4.4}.
Now assume $p > 0$ and $[k : k^p] < \infty$.
Let $\varphi \colon B \to A'$ be a $k$-algebra homomorphism between $k$-affinoid algebras.
By the proof of \cref{reduced-case}, we may assume without loss of generality that $B = k\{r^{-1}T\}$ for some $r = (r_1,\ldots,r_n)$ of positive real numbers.
Let $A := (A')^{\mathrm{red}}$ endowed with the quotient seminorm.
Then $A$ is a reduced $k$-affinoid algebra, and the composition
\[
B=k\{r^{-1}T\} \xrightarrow{\varphi} A' \twoheadrightarrow A
\]
is bounded by \cref{reduced-case}.
By \cref{F-finite-case}, we have $\Omega_{k\{r^{-1}T\}/k[T]} = 0$, and thus the map $\varphi$ is bounded by \cref{deformation-prob}.
This completes the proof.
\end{proof}

\bibliographystyle{skalpha}
\bibliography{bibliography.bib}

\end{document}